\documentclass{article}
\usepackage{amssymb,amsmath,bm,geometry,graphics,graphicx,url,color}

\def\no{\noindent}
\def\pmatrix{\left(\begin{array}}
\def\endpmatrix{\end{array}\right)}

\newtheorem{theo}{Theorem}

\newtheorem{rem}{Remark}
\newtheorem{defi}{Definition}

\newtheorem{prop}{Proposition}
\newtheorem{assum}{Assumption}

\newcommand{\norm}[1]{\left\Vert#1\right\Vert}

\title{Long-term analysis of   exponential integrators  for highly oscillatory conservative systems}

\author{Bin Wang\,
\footnote{School of Mathematical Sciences, Qufu Normal University,
Qufu 273165, P.R. China; Mathematisches Institut, University of
T\"{u}bingen, Auf der Morgenstelle 10, 72076 T\"{u}bingen, Germany.
The research is supported in part by the Alexander von Humboldt
Foundation and by the Natural Science Foundation of Shandong
Province (Outstanding Youth Foundation) under Grant ZR2017JL003.
E-mail:~{\tt wang@na.uni-tuebingen.de} } \and Jiyong
Li\thanks{College of Mathematics and Information Science, Hebei
Normal University, Shijiazhuang   050024, P.R.China E-mail:~{\tt
ljyong406@163.com}} \and Yonglei Fang\thanks{School of Mathematics
and Statistics, Zaozhuang University, Zaozhuang   277160, P.R.China
E-mail:~{\tt ylfangmath@163.com}}
 }

\begin{document}
\maketitle

\begin{abstract}In this paper, we investigate the long-time near-conservations of
energy and kinetic energy by the widely used exponential integrators
to highly oscillatory conservative systems. The modulated Fourier
expansions of two kinds of exponential integrators have been
constructed and the long-time numerical conservations of energy and
kinetic energy are obtained by deriving two almost-invariants of the
expansions. Practical examples of the  methods are given and the
theoretical results are confirmed and demonstrated by a numerical
experiment.
\medskip

\no{\bf Keywords:}  highly oscillatory conservative systems,
modulated Fourier expansion,  exponential integrators, long-time
  energy conservation
\medskip
\no{\bf MSC:}65P10, 65L05

\end{abstract}

\section{Introduction}\label{intro}
In this paper, we are  concerned with the long-term analysis of
implicit exponential integrators for solving the systems of   the
form
\begin{equation}\label{IVPPP}
y^{\prime}(t)=Q \nabla H(y(t)),\quad
y(0)=y_{0}\in\mathbb{R}^{d},\quad t\in[0,T],
\end{equation}
where $Q$ is a $d \times d$ skew symmetric matrix, and $H:
\mathbb{R}^{d}\rightarrow\mathbb{R}$ is defined  by
\begin{equation}\label{H}
 H(y)=  \dfrac{1}{2}y^{\intercal}\big(\frac{1}{\epsilon}M\big)y+V(y).
\end{equation}
Here $\epsilon$ is a small parameter satisfying $0<\epsilon\ll 1$,
$M$ is a $d \times d$ symmetric real matrix, and $V:
\mathbb{R}^{d}\rightarrow\mathbb{R}$ is a differentiable function.
It is important to note that since $Q$ is skew symmetric, the system
\eqref{IVPPP}   is a conservative system with the first integral
$H$: i.e.,
$$H(y(t))\equiv H(y_0)\qquad \textmd{for}\ \    \textmd{any}\ \  t\in[0,T].$$
The kinetic energy of the system \eqref{IVPPP} is given by
$$K(y)=\dfrac{1}{2}y^{\intercal}\big(\frac{1}{\epsilon}M\big)y.$$

For brevity,   by letting
\begin{equation*}\label{AG}
\Omega=\frac{1}{\epsilon}QM,\ \ g(y(t))=Q\nabla V(y(t)),
\end{equation*}
 the system \eqref{IVPPP} can be     rewritten as
\begin{equation}\label{IVP1}
y^{\prime}(t)=\Omega y(t)+g(y(t)),\quad y(0)=y_{0}\in\mathbb{R}^{d}.
\end{equation}
It is well known that the exact solution of \eqref{IVPPP} or
\eqref{IVP1} can be represented by the variation-of-constants
formula
  \begin{equation}
 y(t)= e^{t\Omega}y_0+  t\int_{0}^1 e^{(1-\tau)t\Omega}g(y(\tau t))d\tau.\\
\label{VOC}%
\end{equation}
In the analysis of this paper, it is assumed that the matrix
$\Omega$ is symmetric negative definite or skew-Hermitian with
eigenvalues of large modulus. Under these conditions,  the
exponential $e^{t\Omega}$ enjoys favourable properties such as
uniform boundedness, independent of the time step $t$ (see
\cite{Hochbruck2010}).

  The highly oscillatory system \eqref{IVP1} often arises
 in a wide range of applications such as in engineering,
astronomy,  mechanics,  physics and molecular dynamics (see, e.g.
\cite{hairer2006,Hochbruck2010,wang2017-JCM,wubook2018,wu2013-book}).
There are also some semidiscrete PDEs  such as semilinear
Schr\"{o}dinger equations   fit  this form.
 In recent decades,  as   an efficient approach to
integrating \eqref{IVP1},  exponential  integrators have been widely
 investigated  and developed,  and the reader is   referred to
 \cite{Butcher09,Calvo-2006,Celledoni-2008,Grimm06,Hochbruck2005,Hochbruck2009,Mei2017,Li_Wu(sci2016),wang-2016,IMA2018,wu2017-JCAM,wang2017-Cal}
for example.  A systematic survey of exponential integrators is
referred to \cite{Hochbruck2010}. One important advantage of
exponential integrators is that they make well use of the
variation-of-constants formula \eqref{VOC}, and can  performance
very well even for highly oscillatory problems.

On the other hand,  an important aspect in the numerical simulation
of conservative systems is the approximate conservation of the
invariants  over long times. In order to study the long-time
behaviour for  numerical methods/differential equations,   modulated
Fourier expansion was firstly developed  in \cite{Hairer00}. In the
recent two decades, this  technique has been used successfully  in
the long-time analysis for various numerical methods, such as for
trigonometric integrators in
\cite{Cohen15,Cohen03,hairer2006,MFE2018}, for an implicit-explicit
method in \cite{McLachlan14siam,Stern09},  for heterogeneous
multiscale methods in \cite{Sanz-Serna09} and for splitting methods
in \cite{Gauckler17,Gauckler10-1}. Sofar modulated Fourier expansion
has been presented and developed   as an important mathematical tool
in the long-time analysis (see, e.g.
\cite{Cohen12,Cohen15,Cohen08-1,Gauckler13,Hairer16}).
 However,   for the well known exponential  integrators,  the technique of modulated
Fourier expansions has   only been used  in the long-time analysis
for cubic Schr\"{o}dinger equations (see \cite{Cohen12}). It is
noted that, until now, the long-time analysis of exponential
integrators for Hamiltonian ordinary differential equations has not
been considered in the literature, which motivates this paper.

With this promise,  in this paper, we   present the long-time
analysis of implicit exponential integrators for solving the  highly
oscillatory conservative system \eqref{VOC}.   The technique of
modulated Fourier expansions will be used as an important  tool in
the analysis.  This seems to be the first long-time result for
exponential integrators of Hamiltonian ordinary differential
equations.

We organize the rest of this paper   as follows. In Section
\ref{sec: methods},  two kinds of exponential integrators are
considered for solving  \eqref{VOC} and  an illustrative numerical
experiment is presented to show the long-time behaviour of these
methods. Then in Section \ref{sec: Long-time of scheme 1} we derive
the modulated Fourier expansion for the first class of integrators
and then obtain the long-time near conservations of energy and
kinetic energy by showing two almost-invariants. The analyses of
long time conservations for the second class of exponential
integrators
 are given in Section \ref{sec: Long-time
of scheme 2}.    Section \ref{sec:conclusions} includes the
conclusions of this paper.

\section{Exponential integrators and numerical experiment}\label{sec: methods}

\subsection{Two kinds of methods} In order to solve
\eqref{IVP1} effectively,  exponential integrators are considered
throughout this paper.

\begin{defi}
\label{scheme0} (See \cite{Hochbruck2010}).  An $s$-stage
exponential
integrator for solving  \eqref{IVP1} is given by%
\begin{equation}\label{EI0}
\left\{\begin{array}[c]{ll} &Y^{n+c_i}=e^{c_i h
\Omega}y^n+h\sum_{j=1}^{s}
 a_{ij}(h \Omega)g(Y^{n+c_j}),\qquad i=1,\ldots,s,\\
&y^{n+1}=e^{ h \Omega}y^n+h\sum_{i=1}^{s}
 b_{i}(h \Omega)g(Y^{n+c_i}),
\end{array}\right.
\end{equation}
where   $h$ is  a  stepsize,    $c_i\in[0,1]$ for $i=1,\ldots,s$ are
real constants, and $b_{i}(h \Omega)$ and $a_{ij}(h \Omega)$ for
$i,j=1,\ldots,s$ are matrix-valued and bounded functions of $h
\Omega$.
 The coefficients of
this  exponential integrator can be compactly arranged as a Butcher
Tableau
\[%
\begin{tabular}
[c]{c|c}%
$c$ & $A $\\\hline & $\raisebox{-1.3ex}[0pt]{$b^T $}$
\end{tabular}
=
\begin{tabular}
[c]{c|ccc}%
$c_1$ & $a_{11}(h \Omega)$ & $\cdots$& $a_{1s}(h \Omega)$ \\
$\vdots$ & $\vdots$ & $\vdots$& $\vdots$ \\
$c_s$ & $ a_{s1}(h \Omega)$ &   $ \cdots$&   $ a_{ss}(h \Omega)$ \\
\hline & $\raisebox{-1.3ex}[0.5pt]{$b_{1}(h \Omega)$}$  &
$\raisebox{-1.3ex}[0.5pt]{$\cdots$}$&
$\raisebox{-1.3ex}[0.5pt]{$b_{s}(h \Omega)$}$
\end{tabular}
\]
\end{defi}

As the first example,  approximating the integral in \eqref{VOC}
leads to the following exponential integrator.
\begin{defi}
\label{scheme 1}  An
exponential integrator for solving \eqref{VOC} is defined by%
\begin{equation}\label{EI-T} y^{n+1}=e^{ h \Omega}y^n+\frac{h}{2}
\big(g(y^{n+1})+e^{ h \Omega}g(y^{n})\big).
\end{equation}
This integrator is symmetric and can be considered as a two-stage
exponential integrator with the following Butcher Tableau
\[%
\begin{tabular}
[c]{c|cc}%
$0$ & $0$ & $0$ \\
$1$ & $ \frac{1}{2}e^{ h \Omega}$ &   $ \frac{1}{2}$\\\hline &
$\raisebox{-1.3ex}[0.5pt]{$ \frac{1}{2}e^{ h \Omega}$}$  &
$\raisebox{-1.3ex}[0.5pt]{$ \frac{1}{2}$}$
\end{tabular}
\]
We denote it  by EI-T.
\end{defi}

Besides this integrator, in this paper,   we  also consider
one-stage implicit exponential integrators, which are given as
follows.
\begin{defi}
\label{scheme 2}  An one-stage
implicit exponential integrator   is defined by%
\begin{equation}\label{im integ one-stage}
\left\{\begin{array}[c]{ll} &Y^{n+c_1}=e^{c_1 h \Omega}y^n+h
 a_{11}(h \Omega)g(Y^{n+c_1}),\\
&y^{n+1}=e^{ h \Omega}y^n+h
 b_{1}(h \Omega)g(Y^{n+c_1}).
\end{array}\right.
\end{equation}
This integrator is denoted by EI-O.
\end{defi}

 \renewcommand\arraystretch{1.6}
\begin{table}[!htb]$$
\begin{array}{|c|c|c|c|c|c|c|c|}
\hline
\text{Integrators} &c_1  &a_{11}(h\Omega)   &b_1(h\Omega)  & \text{Symmetric}  &  \text{Reversible}  &\text{Symplectic} \\
\hline
\text{EI-O1} & \frac{1}{2} & \frac{1}{2}& e^{\frac{1}{2}h \Omega}  & \text{Yes}  &\text{Yes} &\text{Yes}  \cr
 \text{EI-O2}   & \frac{1}{2} &\frac{1}{2} & \varphi_1(h \Omega)
& \text{Non} & \text{Non} &\text{Non}\cr
\text{EI-O3} & \frac{2}{3} & \frac{1}{2} &e^{\frac{1}{3}h \Omega}   & \text{Non}  &\text{Non} &\text{Yes}\cr
\text{EI-O4} & \frac{1}{2} & \frac{1}{2}b_1(h\Omega/2) &\varphi_1(h \Omega)   & \text{Yes}  &\text{Yes} &\text{Non}\cr
\text{EI-O5} & \frac{1}{2} & \frac{1}{3}& \frac{2}{3}e^{\frac{1}{2}h \Omega}  & \text{Non}  &\text{Non} &\text{Yes}  \cr
 \hline
\end{array}
$$
\caption{Five one-stage implicit exponential integrators.}
\label{praEI}
\end{table}

Five EI-O integrators  are
 listed  in Table \ref{praEI} and it follows from
 \cite{Cohen12}  that EI-O1 and EI-O4 are both
symmetric and reversible, and the others are neither symmetric nor
reversible. About the symplecticness,
 the authors in   \cite{Mei2017} proved that if a   Runge--Kutta (RK) method with the
coefficients $c_i,\bar{b}_{i}, \bar{a}_{ij}$ is symplectic, then the
   exponential integrator of the coefficients
\begin{equation}
\begin{aligned} a_{ij}(h \Omega)=\bar{a}_{ij}e^{(c_i-c_j) h \Omega},\ \ b_{i}(h \Omega)=\bar{b}_{i}e^{(1-c_i)h \Omega}
\end{aligned}
\label{rev co}%
\end{equation}
is symplectic.  We note that the integrator EI-T  can be written as
a two-stage exponential integrator satisfying \eqref{rev co} and
with
\[%
\begin{tabular}
[c]{c|c}%
$c$ & $\bar{A} $\\\hline & $\raisebox{-1.3ex}[0pt]{$\bar{b}^T $}$
\end{tabular}
=
\begin{tabular}
[c]{c|cc}%
$0$ & $0$ & $0$ \\
$1$ & $ \frac{1}{2}$ &   $ \frac{1}{2}$\\\hline &
$\raisebox{-1.3ex}[0.5pt]{$ \frac{1}{2}$}$  &
$\raisebox{-1.3ex}[0.5pt]{$ \frac{1}{2}$}$
\end{tabular}
\]
This shows that the integrator EI-T is not symplectic by considering
that the trapezoidal rule is not symplectic.
 In the light of the symplecticness condition of one-stage
RK method, one gets $\bar{b}_{1}=2\bar{a}_{11}$. Therefore, a class
of one-stage implicit symplectic exponential integrators is given by
\begin{equation}
\begin{aligned} a_{11}(h \Omega)=\bar{a}_{11},\ \ b_{1}(h
\Omega)=2\bar{a}_{11}e^{(1-c_1)h \Omega}.
\end{aligned}
\label{rev cond}%
\end{equation}
With this result, the properties of symplecticity  are shown in
Table \ref{praEI}.

\subsection{Numerical experiments}
 As an
example, we apply these  methods   to   the following averaged
system in wind-induced oscillation
\begin{equation*}
\begin{aligned}& \left(
                   \begin{array}{c}
                     x_1 \\
                      x_2 \\
                   \end{array}
                 \right)
'=\frac{1}{\epsilon} \left(
    \begin{array}{cc}
      -\zeta& -\lambda\\
       \lambda & -\zeta \\
    \end{array}
  \right)\left(
                   \begin{array}{c}
                     x_1 \\
                      x_2 \\
                   \end{array}
                 \right)+
\left(
                                                                           \begin{array}{c}
                                                                           x_1x_2\\
\frac{1}{2}(x_1^2-x_2^2)
                                                                           \end{array}
                                                                         \right),
\end{aligned}\end{equation*}
 where $\zeta \geq
0$ is a damping factor and $\lambda$ is a detuning parameter. By
setting $$\zeta = r\cos(\theta),\qquad  \lambda =
r\sin(\theta),\qquad r \geq 0,\qquad  \theta = \pi /2,$$ this system
can be transformed into the scheme \eqref{IVPPP} with
\begin{equation*}
\begin{aligned}
&Q=\left(
    \begin{array}{cc}
      0 & -1 \\
      1 & 0 \\
    \end{array}
  \right),\ \ M=\frac{1}{\epsilon}\left(
                  \begin{array}{cc}
                    r & 0 \\
                    0 & r \\
                  \end{array}
                \right),\ \ V=-\frac{1}{2} \big(x_1x_2^2-\frac{1}{3}x_1^3\big).
\end{aligned}\end{equation*}
The energy of this system  is given by
$$H=\frac{1}{2}\frac{r}{\epsilon}(x_1^2+x_2^2)-\frac{1}{2}\big(x_1x_2^2-\frac{1}{3}x_1^3\big).$$
We choose $r=1,\epsilon=10^{-4}$ and use the  initial values
$x_1(0)=1.1  \sqrt{\epsilon},\ x_2(0)=\sqrt{\epsilon}.$  This
problem is solved in a long interval  $[0,10^6]$ with $h=0.5$.  The
conservations of the energy and the kinetic energy  for different
integrators are presented in Figures \ref{p2}-\ref{p7}.

From these results, it can be observed that EI-T and symplectic EI-O
methods conserve
 the energy and the kinetic energy quite well over a long term. The integrator  EI-O2  does
not conserve the energy and the magnetic moment as well as the
others. For the five EI-O integrators, it seems that the
symplecticness condition plays an important role for the long-time
conservations. We will explain  the good numerical behaviours of
EI-T and EI-O  satisfying symplecticness condition theoretically by
the modulated Fourier expansion of the integrators in the rest of
this paper. For the method EI-O4 which does not satisfy
symplecticness condition, it has a  much better numerical behaviour
than we expect. The theoretical reason for this will be further
studied in future.

\begin{figure}[ptb]
\centering
\includegraphics[width=15cm,height=5cm]{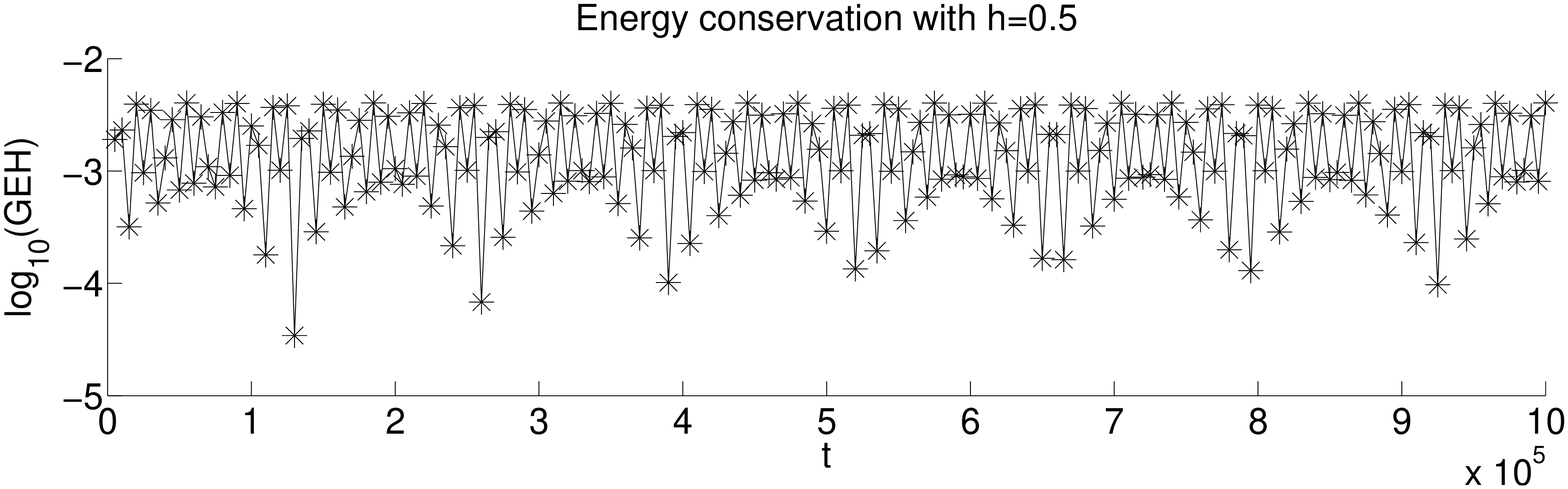}\
\includegraphics[width=15cm,height=5cm]{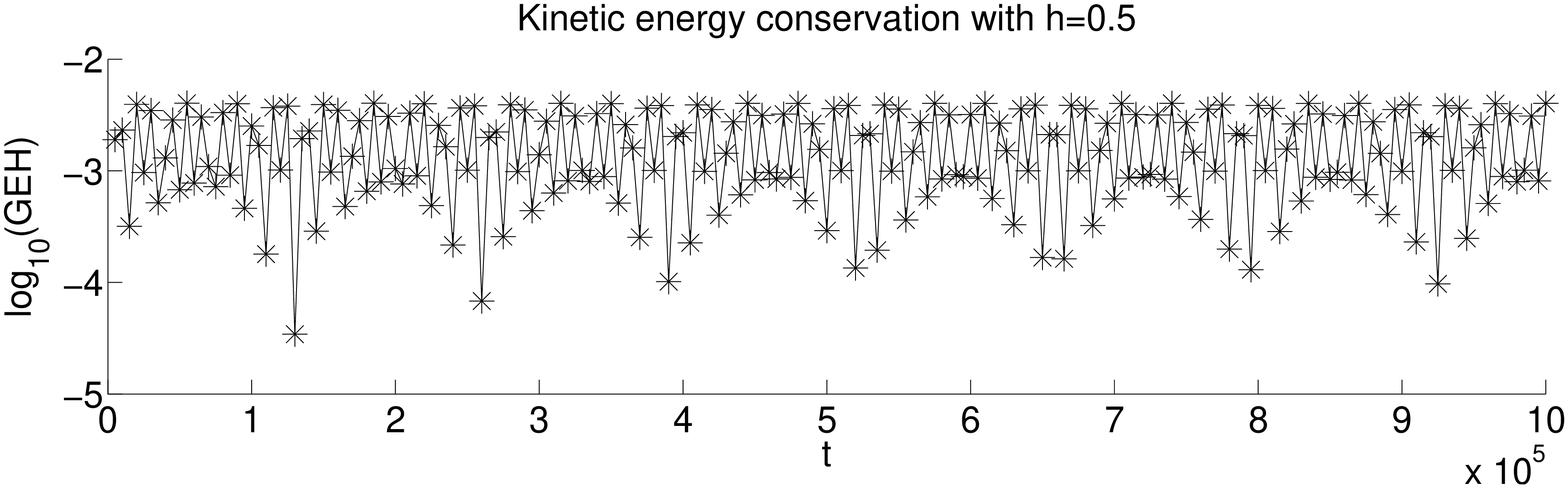}
\caption{EI-T: the logarithm of the  errors against $t$.} \label{p2}
\end{figure}

 \begin{figure}[ptb]
\centering
\includegraphics[width=12cm,height=3cm]{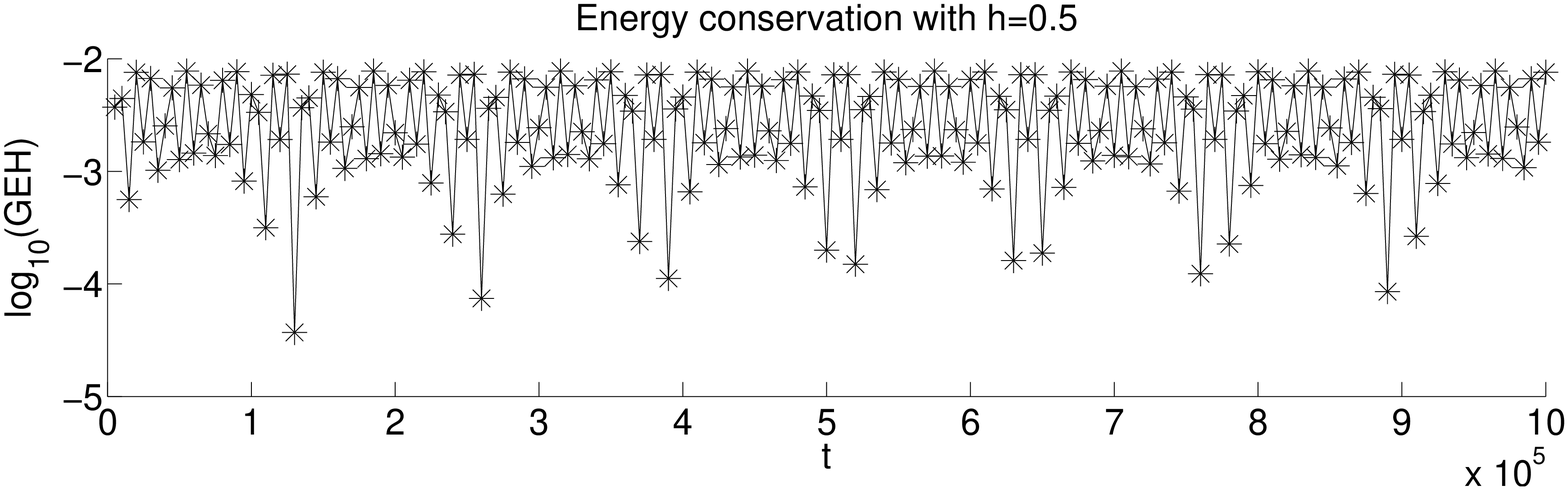}
\includegraphics[width=12cm,height=3cm]{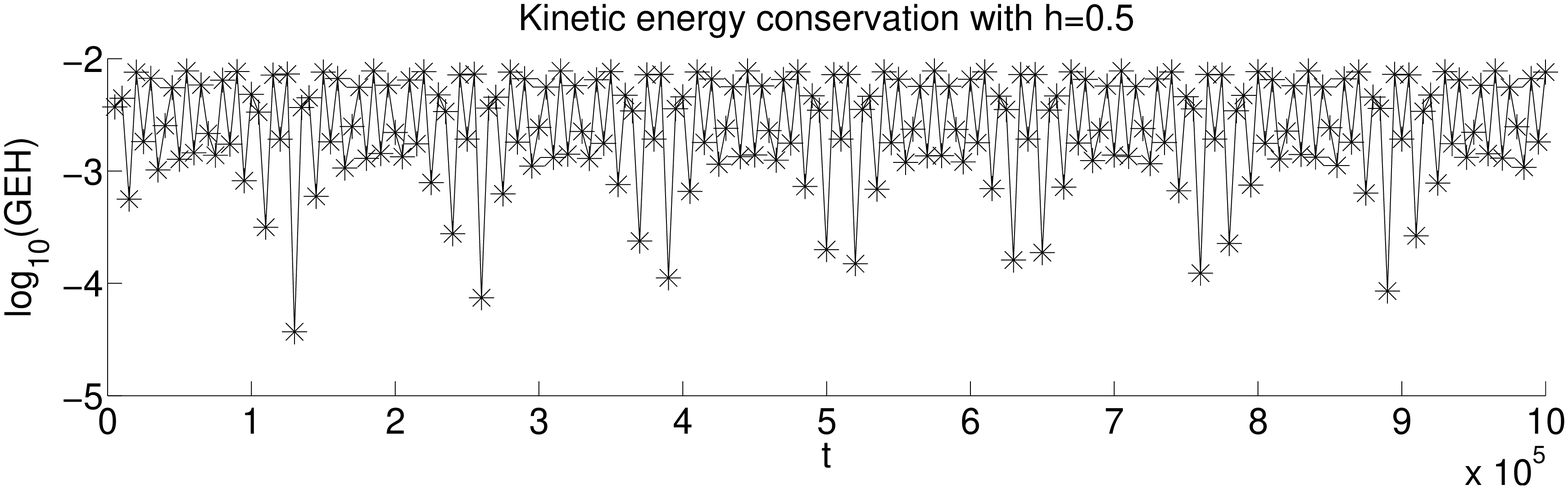}
\caption{EI-O1: the logarithm of the  errors  against  $t$.}
\label{p3}
\end{figure}

 \begin{figure}[ptb]
\centering
\includegraphics[width=12cm,height=3cm]{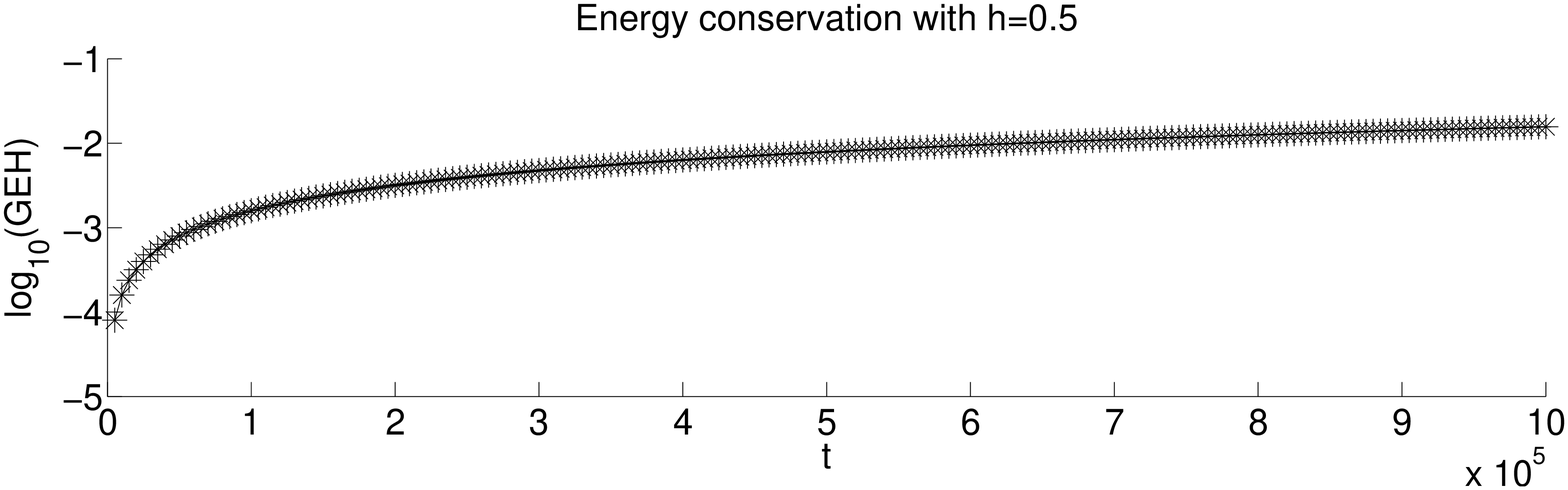}
\includegraphics[width=12cm,height=3cm]{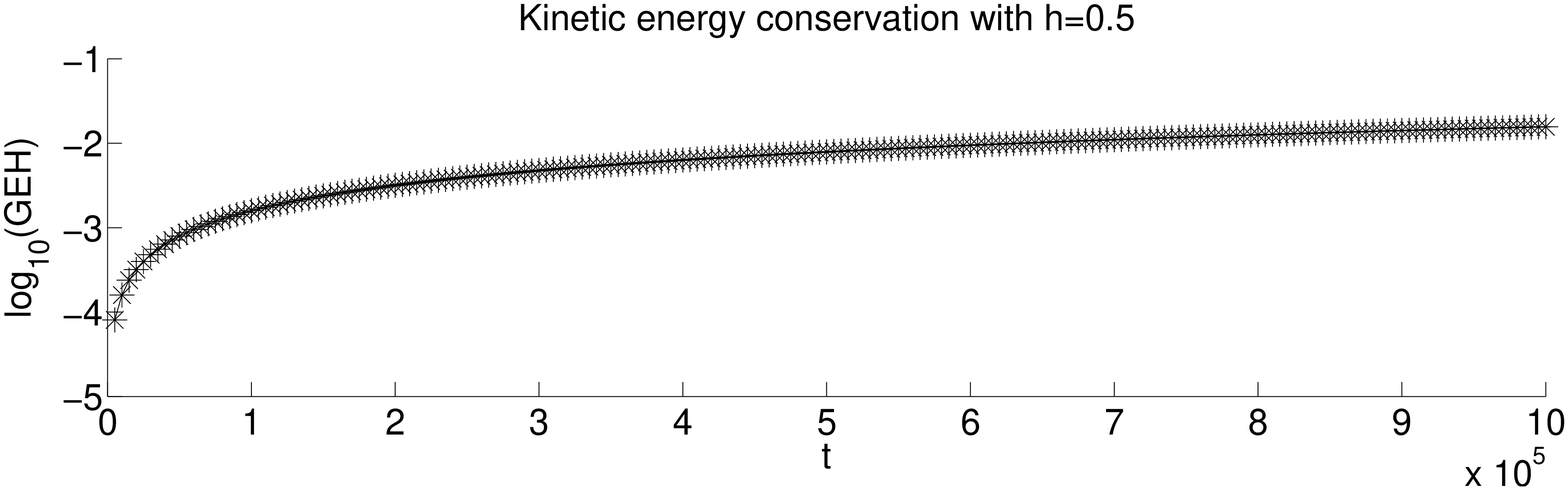}
\caption{EI-O2: the logarithm of the  errors  against  $t$.}
\label{p4}
\end{figure}

 \begin{figure}[ptb]
\centering
\includegraphics[width=12cm,height=3cm]{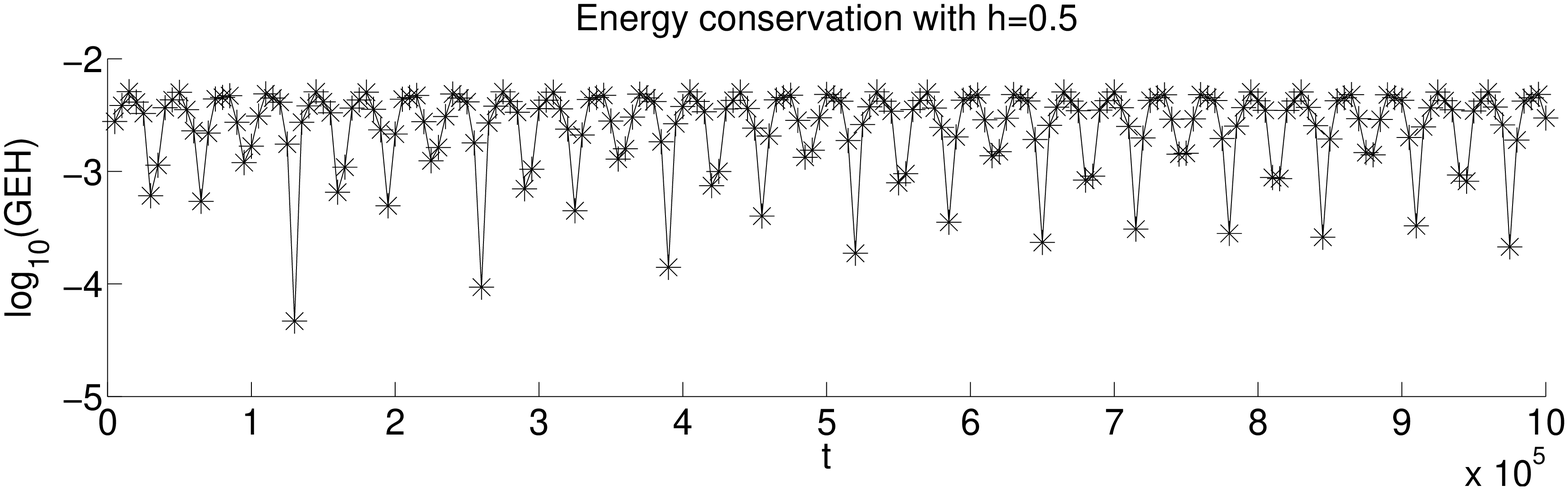}
\includegraphics[width=12cm,height=3cm]{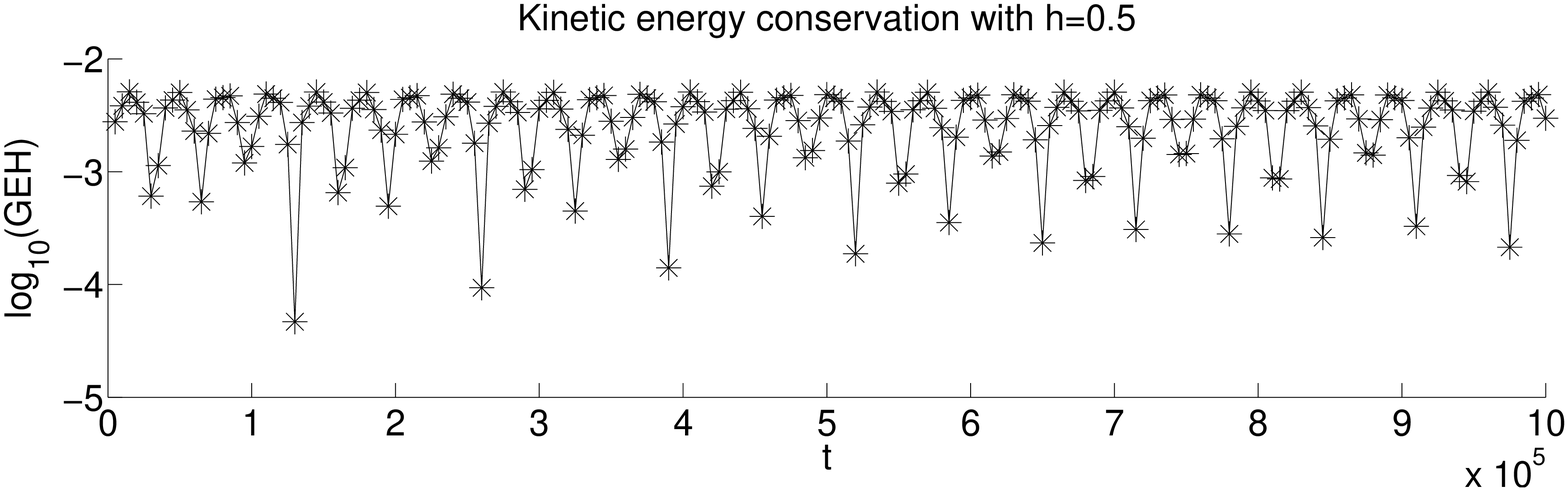}
\caption{EI-O3: the logarithm of the  errors  against  $t$.}
\label{p5}
\end{figure}

 \begin{figure}[ptb]
\centering
\includegraphics[width=12cm,height=3cm]{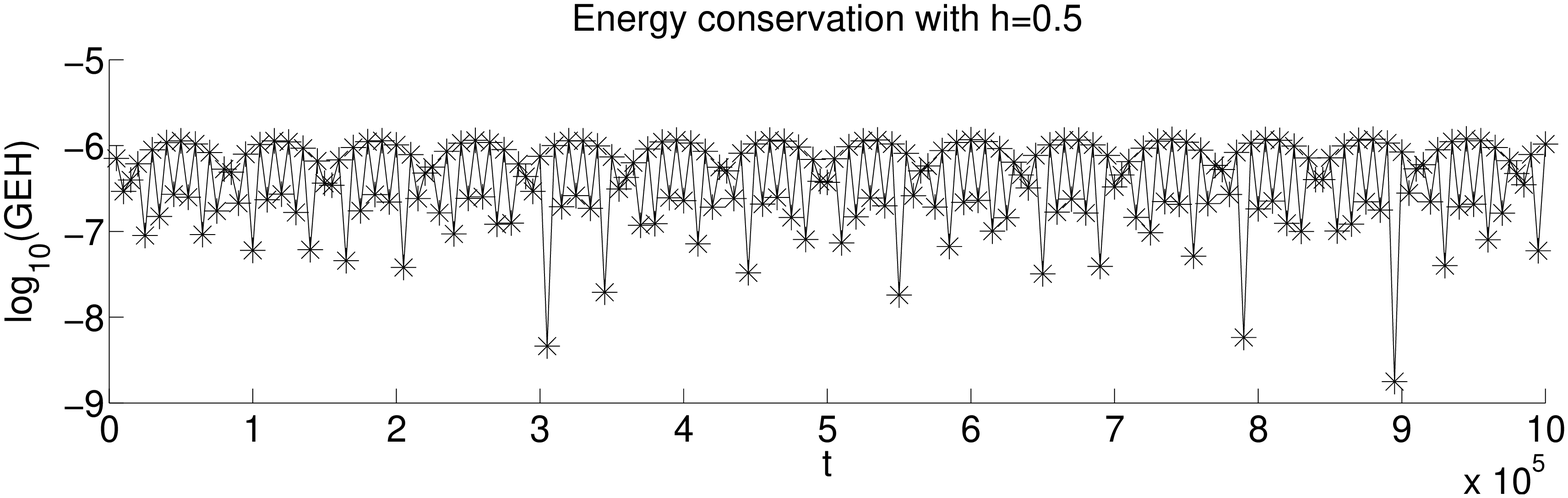}
\includegraphics[width=12cm,height=3cm]{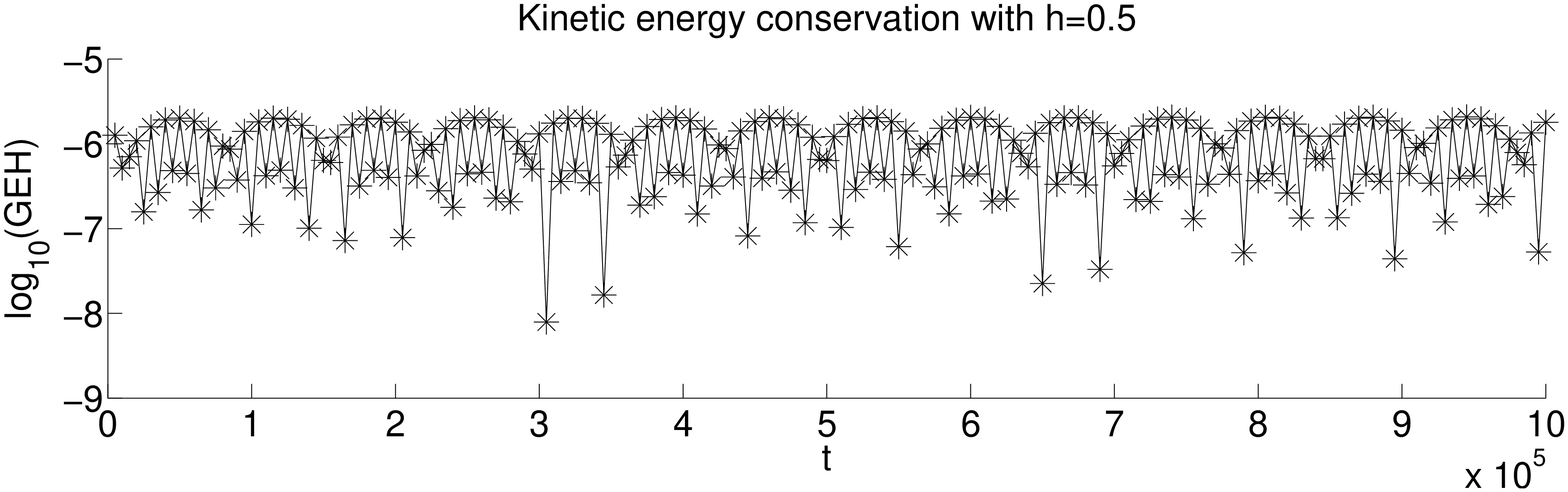}
\caption{EI-O4: the logarithm of the  errors  against  $t$.}
\label{p6}
\end{figure}

 \begin{figure}[ptb]
\centering
\includegraphics[width=12cm,height=3cm]{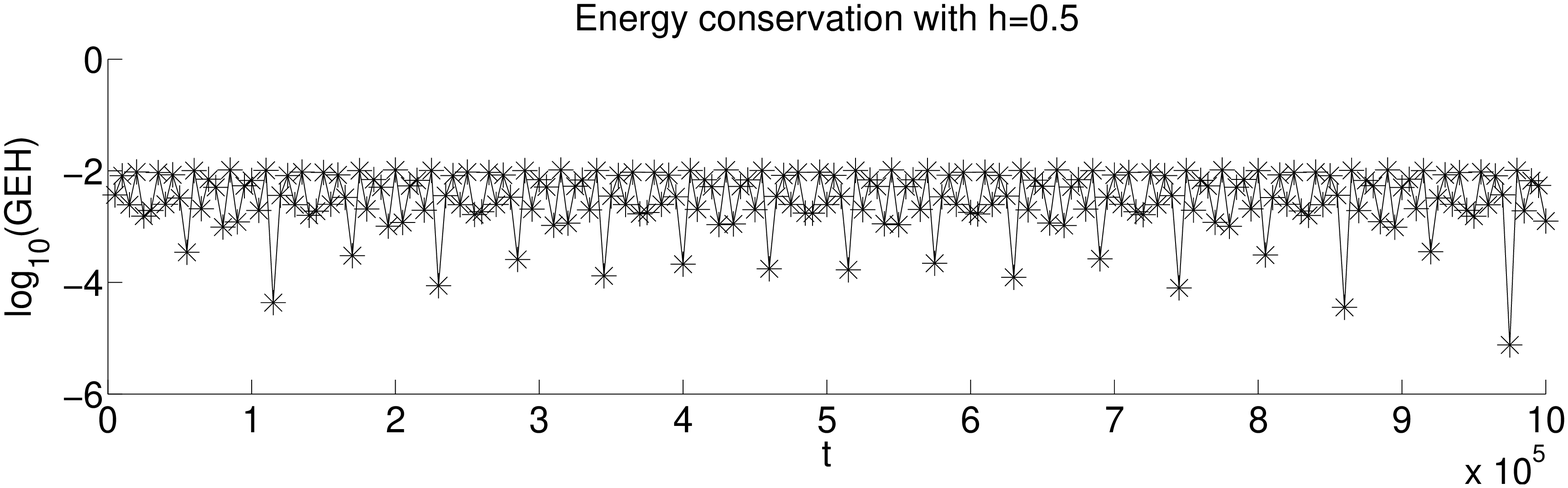}
\includegraphics[width=12cm,height=3cm]{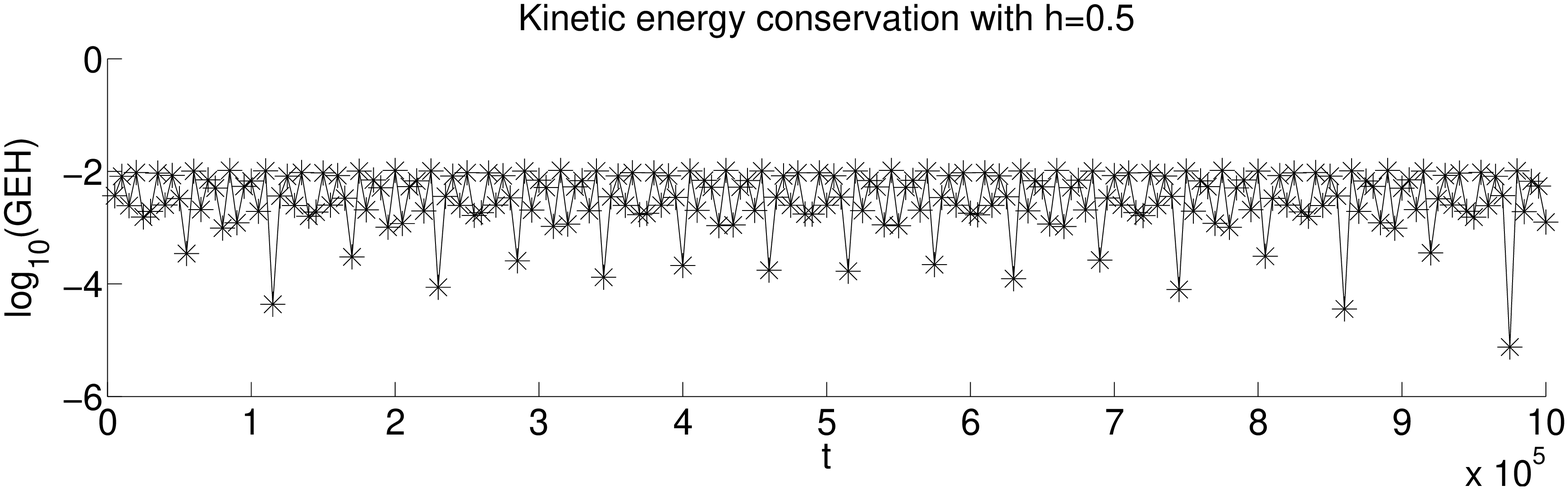}
\caption{EI-O5: the logarithm of the  errors  against  $t$.}
\label{p7}
\end{figure}

\section{Long-time  conservation of the method EI-T}\label{sec: Long-time of
scheme 1} In this section, we show the long-time conservations of
the method EI-T by  modulated Fourier expansion. Our analysis is
limited to the case that $\Omega$ is skew-Hermitian with eigenvalues
of large modulus. The analysis for the  case that $\Omega$ is
symmetric negative definite can be obtained in a same way.

\subsection{Preliminaries}

To do this, we first transform the system  \eqref{IVP1} as follows.
For the skew-Hermitian $\Omega$, there exists a unitary matrix $P$
and a diagonal matrix $\Lambda$ such that $ \Omega=P  \textmd{i}
\Lambda P^\textup{H},$ where
\begin{equation}\label{Lambda}\Lambda=\frac{1}{\epsilon}\textmd{diag}(-\lambda_l
I_{d_l},\ldots,-\lambda_1I_{d_1}, -\lambda_0I_{d_0},\lambda_0
I_{d_0},\lambda_1I_{d_1},\ldots, \lambda_lI_{d_l})\end{equation}
with $\lambda_0=0$ and $\lambda_k>0$. Since
$\Omega=\frac{1}{\epsilon}QM$ with a skew symmetric matrix $Q$ and a
symmetric real matrix $M$, the trance of $\Omega$ is zero. This is
the reason that why $\Lambda$ is assumed to be the form
\eqref{Lambda}.
 With the linear change of variable
  \begin{equation}\label{change of variable}\tilde{y}(t)= P^\textup{H} y(t),\end{equation}the system
\eqref{IVP1} can be rewritten as
\begin{equation}\label{necharged-sts-first order}
\tilde{y}^{\prime}(t)= \textmd{i} \tilde{\Omega}
\tilde{y}(t)+\tilde{g}(\tilde{y}(t)),\quad
\tilde{y}(0)=P^\textup{H}y_{0},
\end{equation}
where $$\tilde{\Omega}=
\textmd{diag}(-\tilde{\omega}_lI_{d_l},\ldots,-\tilde{\omega}_1
I_{d_1},-\tilde{\omega}_0 I_{d_0},\tilde{\omega}_0
I_{d_0},\tilde{\omega}_1I_{d_1},\ldots, \tilde{\omega}_lI_{d_l})$$
with $\tilde{\omega}_k=\frac{ \lambda_k}{\epsilon}$ and
$$\tilde{g}(\tilde{y})=P^\textup{H} g(P \tilde{y})=
-\nabla_{\tilde{y}} V(P\tilde{y}).$$ The energy of this transformed
system is given by
\begin{equation}\label{new energy of cha}
H(y)= \dfrac{1}{2}y^{\intercal}\big(\frac{1}{\epsilon}M\big)y+V(y) =
\dfrac{1}{2}\tilde{y}^{\intercal}\Lambda \tilde{y}+V(P\tilde{y})
:=\tilde{H}(\tilde{y})
\end{equation}
and the kinetic energy    becomes
\begin{equation}\label{newmomentum for B}
K(y)=\dfrac{1}{2}\tilde{y}^{\intercal}\Lambda
\tilde{y}:=\tilde{K}(\tilde{y}).
\end{equation}
For solving  this system,  the EI-T scheme \eqref{EI-T} has the
following form
\begin{equation}\label{ne exp integ one-stage}
\tilde{y}^{n+1}=e^{ \mathrm{i}h
\tilde{\Omega}}\tilde{y}^n+\frac{h}{2}
\big(\tilde{g}(\tilde{y}^{n+1})+e^{\mathrm{i} h
\tilde{\Omega}}\tilde{g}(\tilde{y}^{n})\big).
\end{equation}

 Let
\begin{equation*}
\begin{aligned}&\lambda=(\lambda_1,\ldots,\lambda_l),\quad
k=(k_1,\ldots,k_l),\quad k\cdot
\lambda=k_1\lambda_1+\cdots+k_l\lambda_l,
\end{aligned}
\end{equation*}
and the resonance module is  denoted by
\begin{equation}\mathcal{M}= \{k\in \mathbb{Z}^{l}:\ k\cdot
\lambda=0\}.
\label{M}%
\end{equation} Moreover,   the
following notations will be used in this paper:
\begin{equation*}
\begin{aligned}
&\omega=(\omega_1,\ldots,\omega_l),\quad \langle
j\rangle=(0,\ldots,1,\ldots,0),\quad |k|=|k_1|+\cdots+|k_l|.
\end{aligned}
\end{equation*}
Denote by $\mathcal{K}$   a set of representatives of the
equivalence classes in $\mathbb{Z}^l\backslash \mathcal{M}$ which
are chosen such that for each $k\in\mathcal{K}$ the sum $|k|$ is
minimal in the equivalence class $[k] = k +\mathcal{M},$ and that
with $k\in\mathcal{K}$, also $-k\in\mathcal{K}.$ For the positive
integer $N$, it is denoted that
\begin{equation}\mathcal{N}=\{k\in\mathcal{K}:\ |k|\leq N\},\ \ \ \ \ \mathcal{N}^*=\mathcal{N}\backslash
\{(0,\ldots,0)\}.
\label{mathcalN}%
\end{equation}

In this paper,   the vector $y$ is denoted by
 $$y=(y_{-l},\ldots,y_{-1},y_{-0},y_{0},y_1,\ldots,y_l)$$ with
$q_{\pm j} \in \mathbb{R}^{d_j}$. The same   notation
 is used   for all the
vectors with the same dimension as $y$. Following \cite{hairer2006},
we   define the operator
\begin{equation}\label{L1L2}
 \begin{array}[c]{ll}
L(hD)=(e^{hD}-e^{h
\mathrm{i}\tilde{\Omega}})(e^{hD}+e^{\mathrm{i}h\tilde{\Omega}})^{-1}\end{array}\end{equation}
with the differential operator $D$. We consider the application of
such an operator to functions of the form $\mathrm{e}^{\mathrm{i}(k
\cdot \omega) t}.$
 Furthermore, this operator has  the following
proposition which can be verified easily.
\begin{prop}\label{lhd pro}
 The Taylor expansions of $  L(hD)$   are given by
\begin{equation*}
\begin{aligned}L(hD)= &-\tan(\frac{1}{2}h\tilde{\Omega})\mathrm{i} -(I+\cos(h\tilde{\Omega}))^{-1}
\mathrm{i}(\mathrm{i}hD)-2 \csc^3(h\tilde{\Omega})\sin^4(\frac{1}{2}h\tilde{\Omega})\mathrm{i}(\mathrm{i}hD)^2-\cdots,\\
L(hD+\mathrm{i}h(k\cdot\tilde{\omega}))=&\tan\big(\frac{1}{2}h((k\cdot\tilde{\omega})I-\tilde{\Omega})\big)\mathrm{i}
-\Big(I+\cos\big(h((k\cdot\tilde{\omega})I-\tilde{\Omega})\big)\Big)^{-1}
\mathrm{i}(\mathrm{i}hD)\\
&+2 \csc^3\big(h((k\cdot\tilde{\omega})I-\tilde{\Omega})\big)\sin^4\big(\frac{1}{2}h((k\cdot\tilde{\omega})I+\tilde{\Omega})\big)\mathrm{i}(\mathrm{i}hD)^2+\cdots.\\
\end{aligned}
\end{equation*}
\end{prop}

\subsection{Modulated Fourier expansion}
In this subsection, we derive the modulated Fourier expansion of
EI-T integrator. Before doing that, we need the following
assumptions which have also been used in \cite{Hairer00,Cohen12}.

\begin{assum}\label{ass}
\begin{itemize}

\item  It is assumed that the initial value  $y^0$   satisfies
\begin{equation} \label{initial value energy condition}
 \frac{1}{2\epsilon}\norm{y^{0\intercal} M y^0}^2+V(y^0)\leq E, 
\end{equation}
where $E$ is a constant independent of $\epsilon$.
\item  The numerical solution is supposed to
stay in a compact set on which the potential $V$ is
 smooth.

\item It is required a lower bound    for  the step size
$
h/\epsilon \geq c_0 > 0. 
$

\item The numerical non-resonance condition  is considered
\begin{equation}
|\sin(\frac{h}{2 \epsilon}(k\cdot \lambda))| \geq c \sqrt{h}\ \
\mathrm{for} \ \ k \in \mathbb{Z}^l\backslash \mathcal{M}
   \ \   \mathrm{with} \ \  |k|\leq N\label{numerical non-resonance cond}%
\end{equation}
for some $N\geq2$ and $c>0$.

 \end{itemize}
\end{assum}

\begin{theo}\label{energy thm}
Under the above assumptions   and  for $0 \leq t=nh \leq T$, the
EI-T method \eqref{ne exp integ one-stage} can be expressed by the
following modulated Fourier expansion
\begin{equation}
\tilde{y}^{n}= \tilde{\zeta}(t)+\sum\limits_{k\in\mathcal{N}^*}
\mathrm{e}^{\mathrm{i}(k \cdot \tilde{\omega})
t}\tilde{\zeta}^k(t)+\tilde{R}_{h,N}(t),
\label{MFE-ERKN-0}%
\end{equation}
 where the remainder term  is bounded by
\begin{equation}
 \tilde{R}_{h,N}(t)=\mathcal{O}(th^{N-1}),
\label{remainder}%
\end{equation}
and the coefficient functions as well as all their derivatives are
bounded by
\begin{equation}%
\begin{array}{ll}
 \tilde{\zeta}_{0}(t)=\mathcal{O}(1),\
& \tilde{\zeta}_{\pm j}(t)= \mathcal{O}(\sqrt{h}),\\
 \tilde{\zeta}_{-j}^{-\langle j\rangle}(t)=\mathcal{O}(\sqrt{h}),
\
& \tilde{\zeta}_{j}^{\langle j\rangle}(t)=\mathcal{O}(\sqrt{h})+\cdots,\\
 \tilde{\zeta}_{-
j}^{k}(t)=\mathcal{O}\Big(h^{\frac{|k|+1}{2}}\Big),
 \ k\neq-\langle
j\rangle,\
&\tilde{\zeta}_{j}^{k}(t)=\mathcal{O}\Big(h^{\frac{|k|+1}{2}}\Big),
\ k\neq \langle j\rangle,
\end{array} %
\label{coefficient func}%
\end{equation}
for $j=1,\ldots,l$. It is noted that
$\tilde{\zeta}_{-j}^{-k}=\overline{\tilde{\zeta}_{j}^{k}}$. The
constants symbolised by the notation depend on $E,\ N,\ c_0$ and
$T$, but are independent of $h$ and $\tilde{\omega}$.
\end{theo}
\begin{proof}
In the proof of this theorem, we will construct the function
\begin{equation}
\begin{aligned} &\tilde{y}_{h}(t)= \tilde{\zeta}(t)+\sum\limits_{k\in\mathcal{N}^*}
\mathrm{e}^{\mathrm{i}(k \cdot \tilde{\omega}) t}\tilde{\zeta}^k(t)
\end{aligned}
\label{MFE-1}%
\end{equation} with smooth coefficient functions $\tilde{\zeta}$ and $\tilde{\zeta}^k$ and show   that there is
only a small defect when $\tilde{y}_{h}(t)$ is inserted into the
numerical scheme \eqref{ne exp integ one-stage}.

  $\bullet$ \textbf{Construction of the coefficients
functions.}

Inserting  \eqref{MFE-1}   into \eqref{ne exp integ one-stage} and
using the operator $\mathcal{L}(hD)$ and the Taylor series of the
nonlinearity, we have
\begin{equation*}
\begin{aligned}&\mathcal{L}(hD)\tilde{y}_{h}(t)=\frac{h}{2}
\tilde{g}(\tilde{y}_{h}(t))
=\frac{h}{2}\Big[\tilde{g}(\tilde{\zeta}(t))+\sum\limits_{k\in\mathcal{N}^*}\mathrm{e}^{\mathrm{i}(k
\cdot \tilde{\omega}) t}\sum\limits_{s(\alpha)\sim
k}\frac{1}{m!}\tilde{g}^{(m)}(\tilde{\zeta}(t))(\tilde{\zeta}(t))^{\alpha}
\Big],
\end{aligned} %
\end{equation*}
where the sums   are over all $m\geq1$ and over multi-indices
$\alpha=(\alpha_1,\ldots,\alpha_m)$ with $\alpha_j\in
\mathcal{N}^*$,  and the relation $s(\alpha)\sim k$  means
$s(\alpha)- k\in\mathcal{M}.$ We note that
 an abbreviation for the $m$-tuple
$(\tilde{\zeta}^{\alpha_1}(t),\ldots,\tilde{\zeta}^{\alpha_m}(t))$
is
   denoted by $(\tilde{\zeta}(t))^{\alpha}$.

Inserting  the ansatz \eqref{MFE-1} and comparing the coefficients
of $\mathrm{e}^{\mathrm{i}(k \cdot \tilde{\omega}) t}$  yields
\begin{equation}\label{zeta expre}
\begin{aligned}&\mathcal{L}(hD)\tilde{\zeta}(t)
=\frac{h}{2}\Big[\tilde{g}(\tilde{\zeta}(t))+
\sum\limits_{s(\alpha)\sim
0}\frac{1}{m!}\tilde{g}^{(m)}(\tilde{\zeta}(t))(\tilde{\zeta}(t))^{\alpha} \Big],\\
&\mathcal{L}(hD+\mathrm{i}(k \cdot \tilde{\omega})
h)\tilde{\zeta}^k(t)=\frac{h}{2}\sum\limits_{s(\alpha)\sim
k}\frac{1}{m!}\tilde{g}^{(m)}(\tilde{\zeta}(t))(\tilde{\zeta}(t))^{\alpha}.
\end{aligned} %
\end{equation}
This formula gives the  modulation system  for the coefficients
$\tilde{\zeta}^k(t)$ of the modulated Fourier expansion. According
to Proposition \ref{lhd pro}, the following ansatz of the modulated
Fourier functions $\tilde{\zeta}^k(t)$  can be obtained:
\begin{equation}\label{ansatz}%
\begin{array}{ll}
&\dot{\tilde{\zeta}}_{0}(t)=G_{00}(\cdot)+\cdots,\\
& \tilde{\zeta}_{\pm j}(t)= \frac{\frac{1}{2}h}{-\tan(\pm\frac{1}{2}h\tilde{\omega}_j)\mathrm{i}}\big(G_{\pm j0}(\cdot)+\cdots\big),\\
&\dot{\tilde{\zeta}}_{-j}^{-\langle
j\rangle}(t)=F^{1}_{-j0}(\cdot)+\cdots, \\
& \dot{\tilde{\zeta}}_{j}^{\langle j\rangle}(t)=F^{1}_{j0}(\cdot)+\cdots,\\
& \tilde{\zeta}_{-
j}^{k}(t)=\frac{\frac{1}{2}h}{\tan\big(\frac{1}{2}h(k\cdot\tilde{\omega}+\tilde{\omega}_j)\big)\mathrm{i}}\big(F^k_{-j0}(\cdot)+\cdots\big),
 \ k\neq-\langle
j\rangle,\\
&\tilde{\zeta}_{j}^{k}(t)=\frac{\frac{1}{2}h}{\tan\big(\frac{1}{2}h(k\cdot\tilde{\omega}-\tilde{\omega}_j)\big)\mathrm{i}}\big(F^k_{j0}(\cdot)+\cdots\big),\
k\neq\langle j\rangle,
\end{array} %
\end{equation}
where $j=1,\ldots,l$ and  the dots  stand  for power series in
$\sqrt{h}$.

\vskip1mm  $\bullet$ \textbf{Initial values.}

We  determine the initial values for the differential equations by
considering the conditions that \eqref{MFE-ERKN-0} is satisfied
without remainder term for $t = 0$. From
$\tilde{y}_{h}(t)=\tilde{y}^0$, it follows that
\begin{equation}\label{Initial values-1}%
\begin{aligned}
&\tilde{y}_0^0=\tilde{\zeta}_0(0)+\mathcal{O}(\sqrt{h}),\\
&\tilde{y}_{- j}^0=\tilde{\zeta}_{-j}^{-\langle j\rangle}(0)+\mathcal{O}(\sqrt{h}),\\
&\tilde{y}_{ j}^0=\tilde{\zeta}_{j}^{\langle j\rangle}(0)+\mathcal{O}(\sqrt{h}).\end{aligned} %
\end{equation}
Thus   we get the initial values $\tilde{\zeta}_0(0)$,
$\tilde{\zeta}_{-j}^{-\langle j\rangle}(0)$ and
$\tilde{\zeta}_{j}^{\langle j\rangle}(0)$.

\vskip1mm  $\bullet$ \textbf{Bounds of the coefficients functions.}

The bound \eqref{coefficient func} is immediately obtained on the
base of the above initial values, the ansatz \eqref{ansatz} and
Assumption \ref{ass}.

\vskip1mm  $\bullet$ \textbf{Defect.}

By using the Lipschitz continuous of the nonlinearity  and  the
standard convergence estimates, it is easy to prove the defect
\eqref{remainder}.
\end{proof}

In the light of the linear transform \eqref{change of variable}, the
modulated Fourier expansion for $y^n$ is given as follows.

\begin{theo}\label{energy thm2}
The numerical solution of the EI-T method \eqref{EI-T} admits the
following  modulated Fourier expansion
\begin{equation}
\begin{aligned} &y^{n}=  \zeta(t)+\sum\limits_{k\in\mathcal{N}^*}
\mathrm{e}^{\mathrm{i}(k \cdot \tilde{\omega})
t}\zeta^k(t)+R_{h,N}(t),
\end{aligned}
\label{MFE-ERKN}%
\end{equation}
where $\zeta(t)=P\tilde{\zeta}(t),\ \zeta^k(t)=P\tilde{\zeta}^k(t).$
 The bounds of these functions and the remainders are the same as
 those given in Theorem \ref{energy thm}.
 Moreover, we have $\zeta^{-k}=\overline{\zeta^{k}}$.
\end{theo}

\subsection{Long time energy conservation}
In this subsection, we study long time energy conservation of EI-T
integrator, which is derived by showing an almost-invariant  of the
functions of  modulated Fourier expansions.
\begin{theo}\label{first invariant thm}
 Let
$\vec{\tilde{\zeta}}=\big(\zeta^{k}\big)_{k\in \mathcal{N}}.$ Under
the conditions of Theorem \ref{energy thm}, there exists a function
$\widehat{\mathcal{H}}[\vec{\tilde{\zeta}}]$ such that
\begin{equation}
\widehat{\mathcal{H}}[\vec{\tilde{\zeta}}](t)=\widehat{\mathcal{H}}[\vec{\tilde{\zeta}}](0)+\mathcal{O}(th^{N})
\label{HH}%
\end{equation}
for $0\leq t\leq T.$ Moreover,  the function
$\widehat{\mathcal{H}}[\vec{\tilde{\zeta}}]$ can be expressed as
\begin{equation}\begin{aligned}
\widehat{\mathcal{H}}[\vec{\tilde{\zeta}}]=&\frac{1}{2}
\sum\limits_{j=-l,j\neq0}^l \Big(\tilde{\omega}_j
\big(\tilde{\zeta}_{-j}^{-\langle
j\rangle}\big)^\intercal\tilde{\zeta}_{-j}^{-\langle
j\rangle}+\tilde{\omega}_j \big(\tilde{\zeta}_{j}^{\langle
j\rangle}\big)^\intercal\tilde{\zeta}_{j}^{\langle j\rangle}\Big)
+V(P^\textup{H}\tilde{\zeta}(t))+\mathcal{O}(h).
\label{HH def}%
\end{aligned}
\end{equation}
\end{theo}
\begin{proof}
From the proof of Theorem \ref{energy thm}, it follows that
\begin{equation*}
\begin{aligned}
& L(hD) \tilde{y}^k_{h}(t)=\frac{h}{2}
\tilde{g}(\tilde{y}_{h}(t))+\mathcal{O}(h^{N+1}),
\end{aligned}
\end{equation*}
where we use  the   denotations \begin{equation*}
\begin{aligned}\tilde{y}_{h}(t)=\sum\limits_{k\in\mathcal{N}}\tilde{y}^k_{h}(t)\quad \textmd{with}\quad  \ \tilde{y}^k_{h}(t)=\mathrm{e}^{\mathrm{i}(k \cdot\tilde{\omega})
t}\tilde{\zeta}^k(t).
\end{aligned}
\end{equation*} Multiplication of this result with $P$ yields
\begin{equation*}
\begin{aligned}
& PL(hD)P^\textup{H} P\tilde{y}_{h}(t)=PL(hD)P^\textup{H}
y_{h}(t)\\
=&\frac{h}{2}
P\tilde{g}(\tilde{y}_{h}(t))+\mathcal{O}(h^{N+1})=\frac{h}{2}
g(y_{h}(t))+\mathcal{O}(h^{N+1}),
\end{aligned}
\end{equation*}
where \begin{equation*}
\begin{aligned}y_{h}(t)=\sum\limits_{k\in\mathcal{N}}y^k_{h}(t)\quad \textmd{with}\quad  \ y^k_{h}(t)=\mathrm{e}^{\mathrm{i}(k \cdot\tilde{\omega})
t}\zeta^k(t).
\end{aligned}
\end{equation*}
For the terms of $y^k_{h}$, one gets
\begin{equation}
\begin{aligned} &PL(hD)P^\textup{H} y^k_{h}(t)=-\frac{h}{2}\nabla_{y^{-k}}\mathcal{V}(\vec{y}(t))+\mathcal{O}(h^{N+1}),
\end{aligned}\label{lhdy}%
\end{equation}
where $\mathcal{V}(\vec{y}(t))$ is defined as
\begin{equation}
\begin{aligned}
&\mathcal{V}(\vec{y}(t))=V(y_h^0(t))+
\sum\limits_{s(\alpha)=0}\frac{1}{m!}V^{(m)}(y_h^0(t))
(y_h(t))^{\alpha}
\end{aligned}
\label{newuu}%
\end{equation}
with $$\vec{y}(t)=\big( y_h^{k}(t)\big)_{k\in \mathcal{N}}.$$
 Multiplying \eqref{lhdy} with $
\big(\dot{y}_h^{-k}(t)\big)^\intercal$ and summing up gives
\begin{equation*}
\begin{aligned} &\frac{2}{h}\sum\limits_{k\in\mathcal{N}}
 \big(\dot{y}_h^{-k}(t)\big)^\intercal PL(hD)P^\textup{H}
y^k_{h}(t)+\frac{d}{dt}\mathcal{V}(\vec{y}(t))=\mathcal{O}(h^{N}).
\end{aligned}
\end{equation*}
 By switching  to
the quantities  $\zeta^k$, we obtain
\begin{equation}
\begin{aligned}
 \mathcal{O}(h^{N}) =& \frac{2}{h}\sum\limits_{k\in\mathcal{N}}
 \big(\dot{\zeta}^{-k}(t)-\textmd{i} (k \cdot\tilde{\omega}) \zeta^{-k}(t)\big)^\intercal
P L(hD+\mathrm{i}h(k \cdot\tilde{\omega})) P^\textup{H} \zeta^k(t)+\frac{d}{dt}\mathcal{V}(\vec{\zeta}(t))\\
= &\frac{2}{h}\sum\limits_{k\in\mathcal{N}}
 \big(\dot{ \overline{\zeta^{k}}}(t)-\textmd{i} (k \cdot\tilde{\omega})
 \overline{\zeta^{k}}(t)\big)^\intercal
P L(hD+\mathrm{i}h(k \cdot\tilde{\omega})) P^\textup{H}
\zeta^k(t)+\mathcal{O}(h^{N})\\
= &\frac{2}{h}\sum\limits_{k\in\mathcal{N}}
 \big(\dot{\overline{\tilde{\zeta}}}^k(t)-\textmd{i} (k \cdot\tilde{\omega})
 \overline{\tilde{\zeta}^{k}}(t)\big)^\intercal
P^\textup{H}P L(hD+\mathrm{i}h(k \cdot\tilde{\omega})) P^\textup{H}
P\tilde{\zeta}^k(t)+\mathcal{O}(h^{N})\\
= &\frac{2}{h}\sum\limits_{k\in\mathcal{N}}
 \big(\dot{\overline{\tilde{\zeta}}}^k(t)-\textmd{i} (k \cdot\tilde{\omega})
 \overline{\tilde{\zeta}^{k}}(t)\big)^\intercal
 L(hD+\mathrm{i}h(k \cdot\tilde{\omega})) \tilde{\zeta}^k(t)+\mathcal{O}(h^{N}).
\end{aligned}
\label{duu-new1}%
\end{equation}

By  the Taylor expansions of $  L(hD)$   given in Proposition
 \ref{lhd pro} and the  ``magic formulas" on p. 508
of \cite{hairer2006}, it is easy to check that $\textmd{Im}
\big(\dot{ \overline{\tilde{\zeta}}}^{k}(t)\big)^\intercal
L(hD+\mathrm{i}h(k \cdot\tilde{\omega})) \hat{\zeta}^k(t)$ and
$\textmd{Im} \big(\textmd{i} (k \cdot\tilde{\omega})
\overline{\tilde{\zeta}^{k}}(t)\big)^\intercal L(hD+\mathrm{i}h(k
\cdot\tilde{\omega})) \tilde{\zeta}^k(t)$ are both total
derivatives.
 Therefore,    the imaginary part of the
 right-hand side of
\eqref{duu-new1} is the total derivative.  There exists a function
$\widehat{\mathcal{H}}$ such
 that
$\frac{d}{dt}\widehat{\mathcal{H}}[\vec{\zeta}](t)=\mathcal{O}(h^{N})$
and the statement \eqref{HH} is obtained by an integration.

 The construction \eqref{HH} of $\widehat{\mathcal{H}}$ is shown
by considering the previous analysis and  the bounds of Theorem
\ref{energy thm}.

\end{proof}

The first main result about the long time energy conservation of
EI-T is given as follows.
\begin{theo} \label{no sym Long-time
thm} Under the conditions of Theorem \ref{first invariant thm},  one
obtains
\begin{equation*}
\begin{aligned}
   \mathcal{H}
[\vec{\zeta}](t)&=H(y^n) +\mathcal{O}(h)
\end{aligned}
\end{equation*}
for $0\leq t=nh\leq T.$ Moreover, for the long time energy
conservation of EI-T, we have  \begin{equation*}
\begin{aligned}
H(y^n)&=H(y^0)+\mathcal{O}(h)
\end{aligned}
\end{equation*}
for $0\leq nh\leq h^{-N+1}.$ The constants symbolized by
$\mathcal{O}$ depend on $N, T$ and the constants in the assumptions,
but are independent of $n,\ h,\ \epsilon$.
\end{theo}
\begin{proof}
In the light of   the bounds given in  Theorem \ref{energy thm},  we
deduce that
 \begin{equation}\label{HIE}\begin{aligned} &H(y^n)=\tilde{H}(\tilde{y}^n)=\frac{1}{2}
\sum\limits_{j=-l,j\neq0}^l \Big(\tilde{\omega}_j
\big(\tilde{\zeta}_{-j}^{-\langle
j\rangle}\big)^\intercal\tilde{\zeta}_{-j}^{-\langle
j\rangle}+\tilde{\omega}_j \big(\tilde{\zeta}_{j}^{\langle
j\rangle}\big)^\intercal\tilde{\zeta}_{j}^{\langle j\rangle}\Big)
+V(P^\textup{H}\tilde{\zeta}(t)) +\mathcal{O}(h).
\end{aligned}\end{equation}
A comparison between \eqref{HH def} and \eqref{HIE}  yields the
first result of this theorem. The second statement  of this theorem
is easily obtained by following the same way used in Section XIII of
\cite{hairer2006}.

\hfill
\end{proof}

\subsection{Long time  kinetic energy conservation}
We now turn to the long time  conservation of kinetic energy. Define
the vector functions  of $\vec{\zeta}(\lambda,t)$ as
$$\vec{\zeta}(\lambda,t)=(\mathrm{e}^{\mathrm{i}(k \cdot
\tilde{\omega}) \lambda}\zeta^k(t))_{k\in\mathcal{N}}.$$
 Then it can be observed from the definition \eqref{newuu}
  that $\mathcal{V} (\vec{\zeta}(\lambda,t))$ does  not depend on
$\lambda$. Thus, the following result is obtained
\begin{equation}\label{almost in 2}%
\begin{aligned}0=& \frac{d}{d\lambda}\mid_{\lambda=0}\mathcal{V}
(\vec{\zeta}(\lambda,t))=\sum\limits_{k\in\mathcal{N}}\mathrm{i}(k
\cdot\tilde{\omega})(\zeta^{-k}(t))^\intercal
\nabla_{x^{-k}}\mathcal{V} (\vec{\zeta}(t))\\
= &\frac{2}{-h }\sum\limits_{k\in\mathcal{N}} \mathrm{i}(k
\cdot\tilde{\omega})\big(\zeta^{-k}(t)\big)^\intercal PL(hD+\mathrm{i}h(k \cdot\tilde{\omega}))P^\textup{H}  \zeta^k(t)+\mathcal{O}(h^{N})\\
=&\frac{2}{-h }\sum\limits_{k\in\mathcal{N}} \mathrm{i}(k
\cdot\tilde{\omega}) \big( \overline{\zeta^{k}}(t)\big)^\intercal
PL(hD+\mathrm{i}h(k \cdot\tilde{\omega}))P^\textup{H}
\zeta^k(t)+\mathcal{O}(h^{N})\\
=&\frac{2}{-h }\sum\limits_{k\in\mathcal{N}} \mathrm{i}(k
\cdot\tilde{\omega}) \big(
\overline{\tilde{\zeta}^{k}}(t)\big)^\intercal P^\textup{H}
PL(hD+\mathrm{i}h(k \cdot\tilde{\omega}))P^\textup{H}P
\tilde{\zeta}^k(t)+\mathcal{O}(h^{N})\\
=&\frac{2}{-h }\sum\limits_{k\in\mathcal{N}} \mathrm{i}(k
\cdot\tilde{\omega}) \big(
\overline{\tilde{\zeta}^{k}}(t)\big)^\intercal L(hD+\mathrm{i}h(k
\cdot\tilde{\omega})) \tilde{\zeta}^k(t)+\mathcal{O}(h^{N}).
\end{aligned}
\end{equation}
Similar to the analysis of the above subsection, it can be verified
that the right hand size of \eqref{almost in 2} is a total
derivative. Therefore,      we get the second almost-invariant as
follows.

\begin{theo}\label{2 invariant thm}
Under the conditions of Theorem \ref{energy thm}, for $0\leq t\leq
T$, there exists a function $\widehat{\mathcal{M}}[\vec{\zeta}]$
such that
\begin{equation}
\widehat{\mathcal{M}}[\vec{\tilde{\zeta}}](t)=\widehat{\mathcal{M}}[\vec{\tilde{\zeta}}](0)+\mathcal{O}(th^{N}),
\label{mm}%
\end{equation}
where
\begin{equation}\begin{aligned}
\widehat{\mathcal{M}}[\vec{\tilde{\zeta}}]=\frac{1}{2}
\sum\limits_{j=-l,j\neq0}^l \Big(\tilde{\omega}_j
\big(\tilde{\zeta}_{-j}^{-\langle
j\rangle}\big)^\intercal\tilde{\zeta}_{-j}^{-\langle
j\rangle}+\tilde{\omega}_j \big(\tilde{\zeta}_{j}^{\langle
j\rangle}\big)^\intercal\tilde{\zeta}_{j}^{\langle j\rangle}\Big)
+\mathcal{O}(h).
\label{mm def}%
\end{aligned}
\end{equation}
\end{theo}

 Then, we obtain the   result about the long time  kinetic energy
conservation of EI-T.
\begin{theo} \label{2 sym Long-time
thm} Under the conditions of Theorem \ref{first invariant thm}, we
have
\begin{equation*}
\begin{aligned}
\widehat{\mathcal{M}}[\vec{\tilde{\zeta}}](t)&=K(y^n)
+\mathcal{O}(h)
\end{aligned}
\end{equation*}
for $0\leq t=nh\leq T.$ Moreover,  for the long time kinetic energy
conservation of EI-T, it is true  that
\begin{equation*}
\begin{aligned}
K(y^n)&=K(y^0)+\mathcal{O}(h)
\end{aligned}
\end{equation*}
for $0\leq nh\leq h^{-N+1}.$ The constants symbolized by
$\mathcal{O}$ depend on $N, T$ and the constants in the assumptions,
but are independent of $n,\ h,\ \epsilon$.
\end{theo}

\section{Long-time  conservation of the  EI-O integrators} \label{sec: Long-time
of scheme 2} For solving the transformed system
\eqref{necharged-sts-first order}, the EI-O integrators \eqref{im
integ one-stage} are given as
\begin{equation}\label{new-im integ one-stage}
\left\{\begin{array}[c]{ll} &\tilde{Y}^{n+c_1}=e^{c_1
\mathrm{i}h\tilde{\Omega}}\tilde{y}^n+h
 a_{11}(\mathrm{i}h\tilde{\Omega})\tilde{F}(\tilde{Y}^{n+c_1}),\\
&\tilde{y}^{n+1}=e^{\mathrm{i}h\tilde{\Omega}}\tilde{y}^n+h
 b_{1}(\mathrm{i}h\tilde{\Omega})\tilde{F}(\tilde{Y}^{n+c_1}).
\end{array}\right.
\end{equation}
In this section, we study the long-time  conservations of these
one-stage implicit EI-O integrators.  It is assumed that these
integrators satisfy the condition \eqref{rev cond} in the analysis
of this section.

We start by defining another three operators
\begin{equation}\label{new L1L2}
 \begin{array}[c]{ll}
\hat{L}_1(hD)=(e^{hD}-e^{
\mathrm{i}h\tilde{\Omega}})(e^{\mathrm{i}(1-c_1)h\tilde{\Omega}}e^{c_1hD})^{-1},\\
\hat{L}_2(hD)=(e^{-\mathrm{i}(1-c_1)h\tilde{\Omega}}e^{(1-c_1)hD}+e^{\mathrm{i}
c_1h\tilde{\Omega}}e^{-c_1hD}),\\
 \hat{L}(hD)=(\hat{L}_1\circ
  \hat{L}^{-1}_2\circ  )
(hD).\end{array}\end{equation} It can be checked easily that they
have the following important property.
\begin{prop}\label{new-lhd pro}
For the operator $\hat{L}(hD)$ given in \eqref{new L1L2}, it is true
that   \begin{equation}\label{LEL}\hat{L}(hD)=L(hD),\end{equation}
where $L(hD)$ is defined in \eqref{L1L2}. Therefore, $\hat{L}(hD)$
has the same Taylor series as given in Proposition \eqref{lhd pro}.
\end{prop}

\subsection{Modulated Fourier expansion}
For the EI-O integrator \eqref{new-im integ one-stage},  we assume
that the modulated Fourier expansions  of $\tilde{Y}^{n+c_1}$ and
$\tilde{y}^{n}$ are
\begin{equation}
\begin{aligned}
& \tilde{Y}_h(t+c_1h)=
\tilde{\Upsilon}(t+c_1h)+\sum\limits_{k\in\mathcal{N}^*}
\mathrm{e}^{\mathrm{i}(k \cdot \tilde{\omega})  (t+c_1h)}\tilde{\Upsilon}^k(t+c_1h),\\
 &\tilde{y}_{h}(t)=\tilde{\zeta}(t)+\sum\limits_{k\in\mathcal{N}^*} \mathrm{e}^{\mathrm{i}(k \cdot \tilde{\omega})
t}\tilde{\zeta}^k(t),\
\end{aligned}
\label{3VX}%
\end{equation}
respectively, where $t=nh$.  Considering the scheme of the EI-O
integrator \eqref{new-im integ one-stage}, we have
\begin{equation*}
\begin{array}[c]{ll}
\tilde{Y}^{n+c_1}&=e^{c_1 \mathrm{i}h\tilde{\Omega}}\tilde{y}^n+
 a_{11}(\mathrm{i}h\tilde{\Omega}) b_{1}^{-1}(\mathrm{i}h\tilde{\Omega})(\tilde{y}^{n+1}-e^{ \mathrm{i}h\tilde{\Omega}}\tilde{y}_n)\\
 &=\frac{1}{2}b_{1}^{-1}(\mathrm{i}h\tilde{\Omega})\tilde{y}^{n+1}+(e^{c_1 \mathrm{i}h\tilde{\Omega}} -\frac{1}{2}
 b_{1}^{-1}(\mathrm{i}h\tilde{\Omega})e^{ \mathrm{i}h\tilde{\Omega}})\tilde{y}_n \\
 &=\frac{1}{2} e^{-(1-c_1)\mathrm{i}h\tilde{\Omega}} \tilde{y}^{n+1}+(e^{c_1 \mathrm{i}h\tilde{\Omega}} e^{(1-c_1)\mathrm{i}h\tilde{\Omega}} -\frac{1}{2}
e^{ \mathrm{i}h\tilde{\Omega}})e^{-(1-c_1)\mathrm{i}h\tilde{\Omega}} \tilde{y}_n \\
 &=\frac{1}{2} e^{-(1-c_1)\mathrm{i}h\tilde{\Omega}} \tilde{y}^{n+1}+\frac{1}{2} e^{ c_1\mathrm{i}h\tilde{\Omega}}
 \tilde{y}_n,
\end{array}
\end{equation*}
where the condition \eqref{rev cond} is used here. Inserting the
modulated Fourier expansions into these equations, we obtain
\begin{equation}\label{newreal2}
\begin{array}[c]{ll}
\tilde{Y}_h(t+c_1h)&=\frac{1}{2}
\Big(e^{-(1-c_1)\mathrm{i}h\tilde{\Omega}}
\tilde{y}_h(t+c_1h+(1-c_1)h)+\frac{1}{2} e^{
c_1\mathrm{i}h\tilde{\Omega}} \tilde{y}_h(t+c_1h-c_1h)\Big).
\end{array}
\end{equation}
 Changing the
time from $t+c_1h$ to $t$ yields \begin{equation}\label{new LL REAL
0}
 \tilde{Y}_h(t)=\frac{1}{2}\hat{L}_2(hD)\tilde{y}_h(t),
 \end{equation}
 which leads to
\begin{equation}\label{newreal23}
\begin{array}[c]{ll}
&\tilde{\Upsilon}(t)=\frac{1}{2}\hat{L}_2(hD)\tilde{\zeta}(t),\\
&\tilde{\Upsilon}^k(t)=\frac{1}{2}\hat{L}_2(hD+\mathrm{i}h(k\cdot\tilde{\omega}))\tilde{\zeta}^k(t).
\end{array}
\end{equation}
As an example of this connection, one has
\begin{equation}\label{newreal234}
\begin{array}[c]{ll}
&\tilde{\Upsilon}_{-j}^{-\langle
j\rangle}(t)=\tilde{\zeta}_{-j}^{-\langle
j\rangle}(t)+\mathcal{O}(h),\\
&\tilde{\Upsilon}_{j}^{\langle
j\rangle}(t)=\tilde{\zeta}_{j}^{\langle j\rangle}(t)+\mathcal{O}(h),
\end{array}
\end{equation}
for $j=1,\ldots,l$, which will be used in the next subsection.

On the other hand, by the definition of $\hat{L}_1(hD)$,   the
second equality of \eqref{new-im integ one-stage} can be expressed
as
\begin{equation}\label{new LL REAL}
 \begin{array}[c]{ll}
\hat{L}_1(hD)\tilde{y}_h(t)=h\tilde{F}(\tilde{Y}_h(t)).
 \end{array}\end{equation}
Combining \eqref{new LL REAL 0} with \eqref{new LL REAL} implies
\begin{equation}\label{new LL REAL 2}
 \begin{array}[c]{ll}
\frac{2}{h}\hat{L}(hD)\tilde{Y}_h(t)=\tilde{F}(\tilde{Y}_h(t)).
 \end{array}\end{equation}
Therefore, it is obtained
\begin{equation}\label{3zeta expre}
\begin{aligned}
 &\hat{L}(hD) \tilde{\Upsilon}^0=\frac{h^2}{4}\Big(\tilde{F}(\tilde{\Upsilon}^0)+
\sum\limits_{s(\alpha)\sim
0}\frac{1}{m!}\tilde{F}^{(m)}(\tilde{\Upsilon}^0)(\tilde{\Upsilon})^{\alpha}\Big),\\
 &  \hat{L}(hD+\mathrm{i}(k \cdot \tilde{\omega})h) \tilde{\Upsilon}^k=\frac{h^2}{4}
\sum\limits_{s(\alpha)\sim
k}\frac{1}{m!}\tilde{F}^{(m)}(\tilde{\Upsilon}^0)(\tilde{\Upsilon})^{\alpha},
\end{aligned} %
\end{equation}
which gives   the  modulation system for the coefficients
$\tilde{\Upsilon}^k$.  The modulation system  for the coefficient
 $\tilde{\zeta}^k$ can be obtained by considering \eqref{newreal23}.

\begin{rem}\label{rem}
It can be observed  that the formula \eqref{3zeta expre} is quite
similar to  \eqref{zeta expre}.   Therefore, with the property
\eqref{LEL}, a result similar to Theorem \ref{energy thm} about the
bounds of the coefficient functions  $\tilde{\Upsilon}^k$ can be
obtained. Then the bounds of the coefficient functions
$\tilde{\zeta}^k$ can be derived  by considering \eqref{newreal23}.
Therefore, the modulated Fourier expansions of EI-O integrators
\eqref{new-im integ one-stage} are formulated as follows.
\end{rem}

 \begin{theo}\label{energy thm-5}
 Under the conditions given in
Assumption \ref{ass}  and  for $0 \leq t=nh \leq T$, the EI-O
integrators \eqref{new-im integ one-stage} with the condition
\eqref{rev cond} admit the following modulated Fourier expansions
\begin{equation*}
\begin{aligned}
&\tilde{Y}^{n+c_1}=\tilde{\Upsilon}(t+c_1h)+\sum\limits_{k\in\mathcal{N}^*}
\mathrm{e}^{\mathrm{i}(k \cdot \tilde{\omega}) (t+c_1h)}
\tilde{\Upsilon}^k(t+c_1h)+\mathcal{O}(th^{N-1}),\\
&\tilde{y}^{n}=\tilde{\zeta}(t)+\sum\limits_{k\in\mathcal{N}^*}
\mathrm{e}^{\mathrm{i}(k \cdot \tilde{\omega})
t}\tilde{\zeta}^k(t)+\mathcal{O}(th^{N-1}),
\end{aligned}
\end{equation*}
 where  the coefficient functions $\tilde{\Upsilon}^k$ as well as all their derivatives
 have the same
bounds as \eqref{coefficient func}. The relationship between
$\tilde{\Upsilon}^k$ and $\tilde{\zeta}^k$ is given by
\eqref{newreal23}. For the EI-O integrators \eqref{im integ
one-stage}, their modulated Fourier expansions are given by
\begin{equation*}
\begin{aligned}
& Y^{n+c_1}=\Upsilon(t+c_1h)+\sum\limits_{k\in\mathcal{N}^*}
\mathrm{e}^{\mathrm{i}(k \cdot \tilde{\omega}) (t+c_1h)}
\Upsilon^k(t+c_1h)+\mathcal{O}(th^{N-1}),\\  &y^{n}=
\zeta(t)+\sum\limits_{k\in\mathcal{N}^*} \mathrm{e}^{\mathrm{i}(k
\cdot \tilde{\omega}) t} \zeta^k(t)+\mathcal{O}(th^{N-1}),
\end{aligned}
\end{equation*}
where $t=nh$, $\Upsilon^k=P\tilde{\Upsilon}^k$ and
$\zeta^k=P\tilde{\zeta}^k$.
\end{theo}

\subsection{Long-time  conservation results}

By the same way as stated in Section \ref{sec: Long-time of scheme
1}, we can  derive two almost invariants of the EI-O integrators
\eqref{im integ one-stage}. Based on these results,   the long-time
conservation results can be obtained.  In what follows, we only
present the results   and skip all the proofs for brevity.

\begin{theo}\label{new first invariant thm}
 Letting
$\vec{\tilde{\Upsilon}}= \big( \tilde{\Upsilon}^{k}\big)_{k\in
\mathcal{N}}$ and under the conditions of  Assumption \ref{ass} and
\eqref{rev cond}, there exist two functions
$\widehat{\mathcal{H}}[\vec{\tilde{\Upsilon}}]$ and
$\widehat{\mathcal{M}}[\vec{\tilde{\Upsilon}}]$  such that
\begin{equation}\begin{aligned}
&\widehat{\mathcal{H}}[\vec{\tilde{\Upsilon}}](t)=\widehat{\mathcal{H}}[\vec{\tilde{\Upsilon}}](0)+\mathcal{O}(th^{N}),\\
&
\widehat{\mathcal{M}}[\vec{\tilde{\Upsilon}}](t)=\widehat{\mathcal{M}}[\vec{\tilde{\Upsilon}}](0)+\mathcal{O}(th^{N})
\label{newHH}%
\end{aligned}\end{equation}
for $0\leq t\leq T.$ Moreover,   they  can  be expressed as
\begin{equation}\begin{aligned}
\widehat{\mathcal{H}}[\vec{\tilde{\Upsilon}}]=&\frac{1}{2}
\sum\limits_{j=-l,j\neq0}^l \Big(\tilde{\omega}_j
\big(\tilde{\Upsilon}_{-j}^{-\langle
j\rangle}\big)^\intercal\tilde{\Upsilon}_{-j}^{-\langle
j\rangle}+\tilde{\omega}_j \big(\tilde{\Upsilon}_{j}^{\langle
j\rangle}\big)^\intercal\tilde{\Upsilon}_{j}^{\langle j\rangle}\Big)
+V(P^\textup{H}\tilde{\Upsilon})+\mathcal{O}(h),\\
\widehat{\mathcal{M}}[\vec{\tilde{\Upsilon}}] =&\frac{1}{2}
\sum\limits_{j=-l,j\neq0}^l \Big(\tilde{\omega}_j
\big(\tilde{\Upsilon}_{-j}^{-\langle
j\rangle}\big)^\intercal\tilde{\Upsilon}_{-j}^{-\langle
j\rangle}+\tilde{\omega}_j \big(\tilde{\Upsilon}_{j}^{\langle
j\rangle}\big)^\intercal\tilde{\Upsilon}_{j}^{\langle
j\rangle}\Big)+\mathcal{O}(h).
\label{new HH def}%
\end{aligned}
\end{equation}
In the light of   \eqref{newreal234},  these two almost invariants
can be expressed
\begin{equation}\begin{aligned}
\widehat{\mathcal{H}}=&\frac{1}{2} \sum\limits_{j=-l,j\neq0}^l
\Big(\tilde{\omega}_j \big(\tilde{\zeta}_{-j}^{-\langle
j\rangle}\big)^\intercal\tilde{\zeta}_{-j}^{-\langle
j\rangle}+\tilde{\omega}_j \big(\tilde{\zeta}_{j}^{\langle
j\rangle}\big)^\intercal\tilde{\zeta}_{j}^{\langle j\rangle}\Big)
+V(P^\textup{H}\tilde{\zeta})+\mathcal{O}(h),\\
\widehat{\mathcal{M}} =&\frac{1}{2} \sum\limits_{j=-l,j\neq0}^l
\Big(\tilde{\omega}_j \big(\tilde{\zeta}_{-j}^{-\langle
j\rangle}\big)^\intercal\tilde{\zeta}_{-j}^{-\langle
j\rangle}+\tilde{\omega}_j \big(\tilde{\zeta}_{j}^{\langle
j\rangle}\big)^\intercal\tilde{\zeta}_{j}^{\langle
j\rangle}\Big)+\mathcal{O}(h).
\label{neww HH def}%
\end{aligned}
\end{equation}
\end{theo}
We are now in the position to  present the main results of EI-O
integrators.
\begin{theo} \label{2 sym Long-time
thm} It is assumed that all the conditions of Theorem \ref{new first
invariant thm} are satisfied. Then for the long time energy and
kinetic energy conservations of EI-O integrators, we have
  \begin{equation*}\begin{aligned}
&H(y^n)=H(y^0)+\mathcal{O}(h),\\ &K(y^n)=K(y^0)+\mathcal{O}(h)
\end{aligned}\end{equation*}
for $0\leq nh\leq h^{-N+1}.$ The constants symbolized by
$\mathcal{O}$ depend on $N, T$ and the constants in the assumptions,
but are independent of $n,\ h,\ \epsilon$.
\end{theo}

\section{Conclusions} \label{sec:conclusions}
In this paper, we have studied  the long-time  energy and kinetic
energy near-conservations of
  exponential integrators for
solving  highly oscillatory  conservative systems.
 Two kinds of exponential integrators have been presented and their
modulated Fourier expansions have been developed.  By using the
technique of modulated Fourier expansions,  it is proved   that the
symmetric EI-T and the symplectic EI-O integrators  approximately
conserve the energy and kinetic energy  over long times.

Last but not least, it is noted that we  have tried to derive the
long time result for explicit exponential integrators.
Unfortunately, it does not work since  the  operator  $ L(hD)$
determined  by  explicit exponential integrators does not have good
property. Although implicit exponential integrators need more
computation  in comparison with explicit schemes, they are indeed
used and analysed by many publications (see
\cite{Cano15,Celledoni-2008,Cohen12,Dujardin09}).


\end{document}